\documentclass[11pt]{article}

\usepackage[alphabetic]{amsrefs}
\usepackage{amsmath,amssymb,amsthm,enumerate,color}
\usepackage[all,cmtip]{xy}
\usepackage[applemac]{inputenc}

\def \1{\mathds{1}}
\def \a{\mathfrak a}
\def \al{\alpha}
\def \Ad{\operatorname{Ad}}
\def \bs{\backslash}
\def \C{{\mathbb C}}
\def \CA{{\cal A}}
\def \CB{{\cal B}}

\def \CE{{\cal E}}

\def \CO{{\cal O}}

\def \df{\ \begin{array}{c} _{\rm def}\\ ^{\displaystyle =}\end{array}\ }
\def \der{\mathrm{der}}
\def \diag{\operatorname{diag}}

\def \e{\emph}

\def \g{\mathfrak g}
\def \Ga{\Gamma}
\def \ga{\gamma}
\def \GL{\operatorname{GL}}

\def \Hom{\operatorname{Hom}}

\def \Im{\operatorname{Im}}
\def \Ind{\operatorname{Ind}}
\def \la{\lambda}
\def \La{\Lambda}
\def \m{\mathfrak m}
\def \N{{\mathbb N}}
\def \n{\mathfrak n}
\def \ol{\overline}
\def \om{\omega}

\def \PGL{\operatorname{PGL}}
\def \ph{\varphi}
\def \Q{{\mathbb Q}}
\def \R{{\mathbb R}}
\def \Re{\operatorname{Re}}
\def \reg{\mathrm{reg}}
\def \sm{\smallsetminus}
\def \SL{\operatorname{SL}}
\def \T{{\mathbb T}}
\def \tr{\operatorname{tr}}
\def \vol{\operatorname{vol}}

\def \Z{{\mathbb Z}}
\def \({\left(}
\def \){\right)}

\newcommand{\norm}
[1]{\left|\hspace{-1pt}\left|#1\right|\hspace{-1pt}\right|}

\renewcommand{\sp}
[1]{\left\langle #1\right\rangle}

\newtheorem{lemma}{Lemma}[section]
\newtheorem{proposition}[lemma]{Proposition}
\newtheorem{corollary}[lemma]{Corollary}

\newtheorem{exmples}[lemma]{Examples}

\newtheorem{exmple}[lemma]{Example}

\newtheorem{defi}[lemma]{Definition}
\newenvironment{definition}[0]{\begin{defi}\rm}
{\end{defi}}
\newtheorem{theorem}[lemma]{Theorem}

\begin{document}

\pagestyle{myheadings} \markright{GEOMETRIC ZETA HIGHER RANK p-ADIC}

\title{Geometric zeta functions for higher rank $p$-adic groups\\ \ \\ \small
Illinois J. Math. Vol. 58, Issue , 719-738 (2014)}
\author{Anton Deitmar \& Ming-Hsuan Kang\thanks{This research was supported by NSC grant 100-2115-M-009-008-MY2 (M-H. Kang).
The research was performed
while the first author was visiting the Shing-Tung Yau Center in National Chiao Tung University and 
the National Center for Theoretical
Sciences, Mathematics Division, in Hsinchu, Taiwan. The authors would like
to thank  both institutions for their support and hospitality.}}

\date{}
\maketitle

{\bf Abstract:}
The higher rank Lefschetz formula for p-adic groups is used to prove rationality of a several-variable zeta function attached to the action of a p-adic group on its Bruhat-Tits building.
By specializing to certain lines one gets one-variable zeta functions, which then can be related to geometrically defined zeta functions. 

{\bf MSC primary:} 11M41\\
{\bf secondary:} 11F70,11F72, 11F75, 20E42, 51E24

\tableofcontents

\section*{Introduction}

The zeta functions of Selberg and Ihara are defined by counting closed geodesics in Riemann surfaces and graphs respectively.
We will refer to these and similar zeta functions defined by geometric data as \e{geometric zeta functions}.
Ihara provided in \cite{Ihara} the only known link between geometric and arithmetic zeta functions by showing that 
the Ihara zeta function for a finite graph equals the Hasse-Weil zeta function of the corresponding Shimura curve. 
The crucial step in the proof is to show that the Ihara zeta function equals the characteristic function of the generating Hecke operator, this latter fact being known under the name \e{Ihara formula}.
For higher dimensional buildings, a generalization of Ihara's formula is still outstanding.
For the case of the group $\PGL_3$, this has been provided in 
\cite{KL} and \cite{KLW}.
One problem in higher rank is the lack of a unified zeta function, as in higher dimensional buildings there are several possibilities to generalize Ihara's approach.
One of them is through the trace formula, or rather its specialization, the Lefschetz formula, which, as in the  case of Lie groups, treated in \cite{GAFA}, yields a several variable zeta function.
This paper is devoted to the study of this zeta function and its relation to other geometric zeta functions.

In the first section we recall the Lefschetz formula as proven in \cite{padlef}.
From this we deduce in Section 2 the analytic continuation of the latter and in Section 3 we give a geometric interpretation of this zeta function.
In Section 4 we consider the special case of the group $\PGL_3$ in which case we relate our several variable zeta function to the geometrically defined zeta functions of Kang, Li and Wang in \cite{KLW}.

\section{The Lefschetz formula}\label{sec1}
Let $F$ be a nonarchimedean local field with valuation
ring $\CO$ and uniformizer $\varpi$. 
Let $|\cdot|$ be the absolute value on $F$ normalized by the rule $\mu(xA)=|x|\mu(A)$, where $\mu$ is any additive Haar measure on $F$.
Denote by $G$ a
semisimple linear algebraic group over $F$. 
In the following we will need some facts about reductive $p$-adic groups which can for instance be found in \cite{Tits} and \cite{Cartier}.

Let $K\subset G$ be a
good maximal compact subgroup. Choose a parabolic
subgroup $P=LN$ of $G$ with Levi component $L$. 
Let $A=A_L$ denote the largest split torus in the center of $L$. 
Then
$A$ is called the {\it split component} of $P$. 
There exists a reductive subgroup $M=M_L$ of $L$, containing the derived group $L^\der$, such that $AM$ has finite index in $L$.
Let
$\Phi=\Phi(G,A)$ be the root system of the pair $(G,A)$,
i.e. $\Phi$ consists of all homomorphisms $\alpha :A\to
\GL_1(F)$, written in the form $a\mapsto a^\al$, such that there is $X$ in the Lie algebra of $G$
with $\Ad(a)X= a^\alpha X$ for every $a\in A$. Given
$\alpha$, let $\n_\alpha$ be the Lie algebra generated by
all such $X$ and let $N_\alpha$ be the closed subgroup of
$N$ corresponding to $\n_\alpha$. Let $\Phi^+=\Phi(P,A)$ be
the subset of $\Phi$ consisting of all positive roots with
respect to $P$. Let $\Delta\subset\Phi^+$ be the subset of
simple roots. 
 Let $A^-\subset A$ be the negative Weyl chamber, i.e., $A^-$ is the set of all
$a\in A$ such that $|a^\alpha| <1$ for any $\alpha\in
\Delta$.

An element $g$ of $G$ is called \emph{elliptic} if it is
contained in a compact torus. 
Let $M_{ell}$ denote the set of elliptic elements of $M$.

Let $X^*(A)=\Hom(A,\GL_1)$ be the group of all
homomorphisms as algebraic groups from the torus $A$ to $\GL_1$. This
group is isomorphic to $\Z^r$ with $r=\dim A$. Likewise let
$X_*(A)=\Hom(\GL_1,A)$. There is a natural $\Z$-valued
pairing
\begin{align*}
X^*(A)\times X_*(A) &\to \Hom(\GL_1,\GL_1)\cong\Z\\
(\alpha,\eta) &\mapsto \alpha\circ\eta.
\end{align*}
For every root $\alpha\in\Phi(A,G)\subset X^*(A)$ let
$\breve{\alpha}\in X_*(A)$ be its coroot, i.e., $\breve\al$ is the unique element of $X_*(A)$ such that $A=(\ker(\al))(\Im(\breve\al))$ and $(\alpha,\breve{\alpha})=2$. The valuation $v$ of $F$
gives a group homomorphism $\GL_1(F)\to\Z$. Let $A_c$ be
the unique maximal compact subgroup of $A$. Let $\bar A
=A/A_c$; then $\bar A$ is a $\Z$-lattice of rank $r=r(P)=\dim A$. By composing with the valuation $v$, the group $X^*(A)$
can be identified with
$$
\bar A^*=\Hom(\bar A,\Z).
$$
Let
$$
\a_0^*=\Hom(\bar A,\R)\ \cong\ X^*(A)\otimes\R
$$
be the real vector space of all group homomorphisms from
$\bar A$ to $\R$ and let
$\a^*=\a_0^*\otimes\C=\Hom(\bar A,\C)\cong
X^*(A)\otimes\C$. For $a\in A$ and $\la\in\a^*$ let
$$
a^\la=q^{-\la(a)},
$$
where $q$ is the number of elements in the residue class field of
$F$. In this way we obtain an identification
$$
{\a^*}/\mbox{\small$\frac{2\pi i}{\log q}$}\bar A^* \
\cong\ \Hom(\bar A,\C^\times).
$$
A \e{quasicharacter} of $A$ is a continuous group homomorphism $\nu : A\to\C^\times$. It is called a \e{character} if its image lies in the cricle group $\T=\{z\in\C:|z|=1\}$.
A quasicharacter $\nu$ is called {\it
unramified} if $\nu$ is trivial on $A_c$. The set
$\Hom(\bar A,\C^\times)$ can be identified with the set of
unramified quasicharacters on $A$. Any unramified
quasicharacter $\nu$ can thus be given a unique real part
$\Re(\nu)\in \a_0^*$. This definition extends to general quasicharacters $\chi:A\to\C^\times$
as follows. Choose a splitting $s:\bar A\to A$ of the exact
sequence
$$
1\to A_c\to A\to\bar A\to 1.
$$
Then $\nu=\chi\circ s$ is an unramified character of $A$. Set
$$
\Re(\chi)=\Re(\nu).
$$
This definition does not depend on the choice of the splitting
$s$. For quasicharacters $\chi$, $\chi'$ and $a\in A$ we will
frequently write $a^\chi$ instead of $\chi(a)$ and
$a^{\chi+\chi'}$ instead of $\chi(a)\chi'(a)$. Note that the
absolute value satisfies $|a^\chi|=a^{\Re(\chi)}$ and that a
quasicharacter $\chi$ actually is a character if and only if
$\Re(\chi)=0$.

Let $\Delta_P : P\to\R_+$ be the modular function of the
group $P$. Then the element $\rho=\rho_P=\frac12\sum_{\al\in\Phi^+}\al$, called the \e{modular shift} of $P$, satisfies $\Delta_P(a)=|a^{2\rho_P}|$. For $\nu\in\a^*$ and a root
$\alpha$ let
$$
\nu_\alpha= (\nu,\breve{\alpha})\ \in\
X^*(\GL_1)\otimes\C\ \cong\ \C.
$$
Note that $\nu\in\a_0^*$ implies $\nu_\alpha\in\R$ for
every $\alpha$. For $\nu\in\a_0^*$ we say that $\nu$ is
positive, $\nu>0$, if $\nu_\alpha>0$ for every positive
root $\alpha$.

{\bf Example.} Let $G=\GL_n(F)$ and let $\varpi_j\in G$ be
the diagonal matrix
$\varpi_j=\diag(1,\dots,1,\varpi,1,\dots,1)$ with the
$\varpi$ on the $j$-th position, where $\varpi$ is a uniformizer of the valuation ring $\CO$. Let $\nu\in\a^*$ and let
$$
\nu_j=\nu(\varpi_j A_c)\ \in\ \C.
$$
Let $\alpha$ be a root, say
$\alpha(\diag(a_1,\dots,a_n))=\frac{a_i}{a_j}$. Then
$$
\nu_\alpha= \nu_i-\nu_j.
$$
Hence $\nu\in\a_0^*$ is positive if and only if
$\nu_1>\nu_2>\dots >\nu_n$.

We will fix Haar-measures of $G$ and its reductive subgroups as follows.
For $H\subset G$ being a torus there is a unique maximal compact subgroup $U_H$ which is open. 
Then we fix a Haar measure on $H$ such that $\vol (U_H)=1$.
If $H$ is connected semisimple with compact center then we choose the unique positive Haar-measure which up to sign coincides with the Euler-Poincar\'e measure
\cite{Kottwitz}.
This measure is uniquely determined by the following property.
Any discrete torsion-free cocompact subgroup $\Ga_H\subset H$ has finite dimensional rational cohomology,  i.e., the $\Q$-vector space $H^*(\Ga_H,\Q)=\bigoplus_{p=0}^\infty H^p(\Ga_H,\Q)$ is finite dimensional, where $H^p(\Ga_H,\Q)$ deniotes the group cohomology with coefficients in the field $\Q$.
We denote its Euler characteristic by  $\chi(\Ga_H,\Q)$, i.e.,
$$
\chi(\Ga_H,\Q)=\dim H^{\text{even}}(\Ga_H,\Q)-\dim H^{\text{odd}}(\Ga_H,\Q).
$$
The Euler-Poincar\'e measure on $H$ is the unique Haar measure with the property, that for every 
discrete torsion-free cocompact subgroup $\Ga_H\subset H$ we have
$$
\vol (\Ga_H \bs H) = (-1)^{r(H)}\chi(\Ga_H ,\Q),
$$
where $r(H)$ is the $k$-rank of $H$. For the applications
recall that centralizers of tori in connected groups are connected
\cite{borel-lingroups}.

Assume we are given a discrete subgroup $\Ga$ of $G$ such that the quotient space
$\Ga \bs G$ is compact. Let $(\omega ,V_\omega)$ be a finite dimensional unitary
representation of $\Ga$ and let $L^2(\Ga \bs G,\omega)$ be the Hilbert space
consisting of all measurable functions $f: G \to V_\omega$ such that $f(\ga x) =
\omega(\ga) f(x)$ and $\norm f$ is square integrable over $\Ga \bs G$ (modulo null
functions). Let $R$ denote the unitary representation of $G$ on $L^2(\Ga \bs
G,\omega)$ defined by right shifts, i.e. $R(g) \ph (x) = \ph (xg)$ for $\ph \in
L^2(\Ga \bs G,\omega)$. It is known, that as a $G$-representation this space splits
as a topological direct sum:
$$
L^2(\Ga \bs G,\omega) = \bigoplus_{\pi \in \hat{G}} N_{\Ga ,\omega}(\pi) \pi
$$
with finite multiplicities $N_{\Ga ,\omega}(\pi)<\infty$.

Suppose $\ga\in\Ga$ is $G$-conjugate to some $a_\ga b_\ga\in A^-M_{ell}$.
Let $G_\ga$ and $\Ga_\ga$ denote the centralizers of $\ga$ in $G$ and $\ga$ respectively.
We want to compute the covolume
$$
\vol(\Ga_\ga \bs G_\ga) .
$$
Let $\bar F$ denote an algebraic closure of the ground field $F$. 
An element $g$ of $\GL_n(F)$ is called \e{neat}, if the subgroup of $\bar F^\times$ generated by the eigenvalues of $g$, is torsion-free. 
An element $x$ of $G$ is called 
 \emph{neat} if for  some injective
representation $\rho : G\to GL_n(F)$ of $G$ the matrix $\rho(x)$ is neat.
It is easy to check that in this case the same property holds for every representation $\rho$, injective or not.
A subset $A$ of $G$ is called neat if each element of it
is neat. 
If the characteristic of $F$ is zero, then every arithmetic group
$\Ga$ has a finite index subgroup which is neat \cite{borel}. 

We suppose that $\Ga$ is
neat. Since $\Ga$ is cocompact, this implies that for every
$\ga\in\Ga$ the Zariski closure of the group generated by $\ga$ is a torus. It then
follows  that the centralizer $G_\ga$ is a connected reductive group \cite{borel-lingroups}.

An element $\ga\in\Ga$ is called \emph{primitive} if $\ga =\sigma^n$ with
$\sigma\in\Ga$ and $n\in\N$ implies $n=1$. It is a property of discrete cocompact
torsion-free subgroups $\Ga$ of $G$ that every $\ga\in\Ga$, $\ga\ne 1$ is a positive
power of a unique primitive element. In other words, given a nontrivial $\ga\in\Ga$
there exists a unique primitive $\ga_0$ and a unique $\mu(\ga)\in\N$ such that
$$
\ga =\ga_0^{\mu(\ga)}.
$$

Let $\Sigma$ be a group with finite dimensional rational cohomology.
For $r\in\N$ we define the \emph{higher Euler characteristic} as
$$
\chi_{_r}(\Sigma) = \chi_{_r}(\Sigma ,\Q) \ :=\ \sum_{p=0}^{\infty} 
(-1)^{p+r}\binom pr
\dim H^p(\Sigma ,\Q),
$$
where the sum is actually finite.
As $\Ga$ acts freely on the Bruhat-Tits building $\CB$ of $G$, which is contractible, the quotient $\Ga\bs \CB$ is a classifying space for $\Ga$, hence the rational cohomology of $\Ga$ coincides with the cohomology of the finite CW-complex $\Ga\bs\CB$, hence is finite-dimensional.

 We denote by $\CE_P(\Ga)$
the set of all conjugacy classes $[\ga]$ in $\Ga$ such that $\ga$ is $G$-conjugate to an element $a_\ga m_\ga\in AM$, where $m_\ga$ is elliptic and
$a_\ga\in A^-$.

Let $\ga\in\CE_P(\Ga)$. To simplify the notation let us assume that $\ga=a_\ga
m_\ga\in A^- M_{ell}$. Let $C_\ga$ be the connected component of the center of
$G_\ga$ then
$C_\ga = AB_\ga$, where
$B_\ga$ is the connected center of $M_{m_\ga}$ the latter group will also be
written as $M_\ga$. Let $M_\ga^{der}$ be the derived group of $M_\ga$. Then
$M_\ga=M_\ga^{der} B_\ga$.

Let $\Ga_{\ga,A}=A\cap \Ga_\ga B_\ga$ and $\Ga_{\ga,M}=M_\ga^{der}\cap \Ga_\ga
AB_\ga$. Similar to the proof of Lemma 3.3 of \cite{Wolf}, one shows that
$\Ga_{\ga,A}$ and
$\Ga_{\ga,M}$ are discrete cocompact subgroups of $A$ and $M_\ga^{der}$ respectively.
Let
$$
\la_{\ga}\df \vol(\Ga_{\ga,A}\bs A).
$$

We are now able to express the covolume of the centralizers $\vol(\Ga_\ga\bs G_\ga)$ in terms of higher Euler characteristics.

\begin{proposition}\label{2.3}
\begin{enumerate}[\rm (a)]
\item Assume that $\Ga$ is neat and let $\ga\in\Ga$ be $G$-conjugate to an element of $A^-M_{ell}$. 
Then we obtain
$$
\vol (\Ga_\ga\bs G_\ga) = \la_\ga\ |\chi_{_r}(\Ga_\ga)|,
$$
where $r=\dim A$.
\item
Let $\Ga,\Lambda$ be groups with finite dimensional rational cohomology.
Let $C_r$ be a group isomorphic to $\Z^r$ and assume there is an exact sequence
$$
1\to C_r\to \Ga\to \Lambda\to 1.
$$
Assume that $C_r$ is central in $\Ga$. 
Then 
$$
\chi (\Lambda,\Q) = \chi_{r}(\Ga,\Q).
$$
\end{enumerate}
\end{proposition}

\begin{proof}
\cite{padlef}
\end{proof}

For a representation $(\pi,V_\pi)$ of $G$ let $(\pi^\infty,V_\pi^\infty)$ denote the subrepresentation of
\emph{smooth vectors}, i.e., $\pi^\infty$ is the representation on the space
$\bigcup_{H\subset G} V_\pi^H$, where $H$ ranges over the set of all open subgroups
of $G$ and $V_\pi^H$ is the subspace of $H$-stable vectors. Further let $\pi_N$ denote the \emph{Jacquet module} of $\pi$. By
definition $\pi_N$ is the largest quotient $MAN$-module of $\pi^\infty$ on which
$N$ acts trivially. One can achieve this by factoring out the vector subspace
consisting of all vectors of the form $v-\pi(n)v$ for $v\in\pi^\infty$, $n\in N$. It
is known that if $\pi$ is an irreducible admissible representation, then $\pi_N$ is a
admissible $MA$-module of finite length. For a smooth $M$-module $V$ let
$H_c^\bullet(M,V)$ denote the continuous cohomology with coefficients in $V$ as in
\cite{Borel-Wallach}.

Let $\sigma$ be an element of a group $S$ acting on a finite dimensional  $F$-vector space $V$.
Then we write $\la_{\min}(\sigma\mid V)$ for the minimal norm of an eigenvalue of $\sigma$ in the algebraic closure $\bar F$ of $F$.
Likewise, $\la_{\max}(\sigma\mid V)$ is the maximal norm of such an eigenvalue.
The Lie algebra $\g$ of $G$ has a direct sum decomposition $\g=\bar \n+\m+\n$, where $\m$ is the Lie algebra of $M$ and $\n$ is the Lie algebra of $N$ as well as $\bar\n$ is the Lie algebra of the opposite of $N$.
Then let $\tilde M$ denote the set of all $m\in M$ such that
$$
\la_{\min}(m\mid \bar \n)>\la_{\max}(m\mid \m+\n).
$$ 

\begin{theorem}(Lefschetz Formula)\label{Lefschetz}\\
Let $\Ga$ be a neat discrete cocompact subgroup of $G$.
Let $\ph$ be a uniformly smooth function on $A$ with support in $A^-$. Suppose that
the function $a\mapsto \ph(a)|a^{-2\rho}|$ is integrable on $A$. Let
$\sigma$ be a finite dimensional unitary representation of $M$. Let $q$ be the $F$-split rank
of $G$ and $r=\dim A$. Then
\begin{align*}
\sum_{\pi\in\hat G} N_{\Ga,\omega}(\pi)\sum_{q=0}^{\dim M} (-1)^q \int_{A^-}
\ph(a)\,\tr(a\mid H_c^q(M,\pi_N\otimes\sigma))\, da\\
=
\sum_{[\ga]\in\CE_P(\Ga)} \la_\ga\,
|\chi_r(\Ga_\ga)|\,\tr\omega(\ga)\,\tr\sigma(m_\ga)\,a_\ga^{2\rho}\,\ph(a_\ga).
\end{align*}

Both outer sums converge absolutely and the sum over $\pi\in\hat G$ actually is a
finite sum, i.e., the summand is zero for all but finitely many $\pi$. For a given
compact open subgroup $U$ of $A$ both sides represent a continuous linear functional
on the space of all functions $\ph$ as above which factor over $A/U$, where this
space is equipped with the norm $\norm \ph=\int_A|\ph(a)| a^{2\rho}\, da$.
\end{theorem}

Let $A^*$ denote the set of all continuous group homomorphisms $\la\colon
A\to\C^\times$, which we write in the form $a\mapsto a^\la$. For $\la\in A^*$ and an $A$-module $V$ let $V_\la$ denote the
generalized $\la$-eigenspace, i.e.,
$$
V_\la\df \bigcup_{k=1}^\infty \{ v\in V\mid (a-a^\la)^k v=0\ \forall a\in A\}.
$$
Then
\begin{align*}
&\int_{A^-} \ph(a)\,\tr(a\mid H_c^q(M,\pi_N\otimes\sigma))\, da\\
&=
\sum_{\la\in A^*}\dim
H_c^q(M,\pi_N\otimes\sigma)_\la\,\int_{A^-}\ph(a)\,a^\la\, da.
\end{align*}

For $\la\in A^*$ define
$$
m_\la^{\sigma,\omega} \df \sum_{\pi\in\hat G}N_{\Ga,\omega}(\pi)\sum_{q=0}^{\dim
M}(-1)^q\,\dim H_c^q(M,\pi_N\otimes\sigma)_\la.
$$
The sum is always finite.
Theorem \ref{Lefschetz} is equivalent to the following Corollary.

\begin{corollary}
(Lefschetz Formula)\\
As an identity of distributions on $A^-$ we have
$$
\sum_{\la\in A^*} m_\la^{\sigma,\omega}\, \la= \sum_{[\ga]\in\CE_P(\Ga)}
\la_\ga\,|\chi_r(\Ga_\ga)|\,a_\ga^{2\rho}\,\tr\omega(\ga)\,\tr\sigma(m_\ga)\,\delta_{a_\ga}.
$$
\end{corollary}

The proofs of the theorem and the corollary are in \cite{padlef}.

\section{The zeta function}\label{Sec2}
Let $q$ denote the residue field cardinality, so $q=|\CO/\varpi\CO|$ and let $r=r(P)=\dim A$.
There are uniquely determined positive integer multiples $\al_1,\dots,\al_{r}$ of the simple roots, such that the modular shift can be written as
$$
2\rho_P=\al_1+\dots+\al_{r}.
$$
For $a\in A^-$ we write
$$
l_j(a)=-\log_q(a^{\al_j}),\quad j=1,\dots,r.
$$
Then $l_j(a)$ is a non-negative integer.
For $\ga\in\CE_P(\Ga)$ we also write $l_j(\ga)=l_j(a_\ga)$.
For $u\in\C^n$ we write
$$
u^{L(a)}=u_1^{l_1(a)}\cdots u_r^{l_r(a)}
$$
and likewise $u^{L(\ga)}$.
Note that this several variable expression must not be mixed up with $u^{l(\ga)}$, where $u$ is a single variable and $l(\ga)$ is the length of a geodesic closed by $\ga$.
These two notions agree in the rank one case, but differ in higher rank.

For $u\in\C^r$ consider the series
$$
S_\Ga(u)=S_{\Ga,P,\om,\sigma}(u)=\sum_{[\ga]\in\CE_P(\Ga)}\la_\ga\,|\chi_r(\Ga_\ga)|\,\tr\om(\ga)\tr\sigma(m_\ga)u^{L(\ga)}.
$$
In the setting of Lie groups, the analogue of this function 
has been introduced in \cite{GAFA}, where it is used to establish a higher rank prime geodesic theorem.
Although not immediately clear by its definition, this function actually counts closed geodesics and therefore deserves to be called a "geometric zeta function," as is explained in the next section.

The following theorem gives a generalization of the Selberg zeta function  to higher rank groups. In order to collect all information the Lefschetz formula has to offer, it is necessary to encode it in a function of several variables, the number of variables given by the rank of the split torus.
This function then turns out to be a rational function and it indeed encodes the information of the Lefschetz formula in a neat and handy way.
It also contains all information encoded in geometric zeta functions, like the ones in \cite{KLW}, as we will make explicit in Section \ref{SecPGL}.

\begin{theorem}\label{thm2.1}
The series $S_\Ga(u)$ converges locally uniformly in the set
$$
\{ u\in\C^r:|u_j|< 1/q,\ j=1,\dots,r\}.
$$
It is a rational function in $u$. More precisely, there exists a finite subset $E\subset\bar A$, elements $a_1,\dots, a_r\in\bar A$
and natural numbers $k_1,\dots,k_r$ such that
$$
S_\Ga(u)=\sum_{\la\in\bar A^*}m_{\la+2\rho}^{\sigma,\om}
\frac1{1-a_1^\la u_1^{k_1}}\cdots\frac1{1-a_r^\la u_r^{k_r}}
\(\sum_{v\in E}v^{\la+2\rho}u^{L(v)}\).
$$
Both  sums are finite, so in particular, the coefficient  $m_{\la+2\rho}^{\sigma,\om}$ is zero for almost all $\la\in \bar A^*$.
\end{theorem}

\begin{proof}
Note that $a\mapsto u^{l_j(a)}$ is the restriction of a character on $A$ to $A^-$ which we write as $a\mapsto a^{s_j}$.
Also, we write $a^s$ for $a^{s_1}\cdots a^{s_r}$.

For $u\in\C$ consider the function $\ph_u:A\to \C$ defined by
$$
\ph_u(a)=\begin{cases} \(uq\)^{l(a)}=u^{l(a)}a^{-2\rho}&a\in A^-,\\ 0& a\notin A^-.\end{cases}
$$
The function $\ph$ factors over $\bar A$, therefore is uniformly smooth.
It is easy to see that $\ph_u(a)|a^{-2\rho}|$ is integrable on $A$ if and only if $|u_j|<1/q$ for every $j=1,\dots,r$.
Assuming this, $\ph_u$ satisfies the Lefschetz formula, the geometric side of which equals
$
S_\Ga(u).
$
The spectral side is
$$
\sum_{\la\in A^*}m_\la^{\sigma,\om}\int_A\ph_u(a)a^\la\,da.
$$

\begin{definition}
Let $V$ denote a $\Q$-vector space of dimension $r\in\N$.
Let $V_\R=V\otimes\R$. A subset $C\subset V_\R$ is called a \e{sharp rational open cone} with $r$ sides if there exist $\al_1,\dots,\al_r\in\Hom(V,\Q)$ such that
$$
C=\{ v\in V_\R: \al_1(v)>0,\dots,\al_r(v)>0\}
$$
and its closure $\ol C$ does not contain a line.

\end{definition}

\begin{lemma}
Let $V$ denote a $\Q$-vector space of dimension $r\in\N$ and let $C$ be a sharp rational open cone in $V_\R$.
Let $\Sigma\subset V$ be a lattice, i.e., a finitely generated subgroup which spans $V$.
Then there exists a finite subset $E\subset\Sigma$ and elements $a_1,\dots,a_r\in\Sigma$ such that $C\cap\Sigma$ is the set of all $v\in V$ of the form
$$
v=v_0+\nu_1a_1+\dots+\nu_ra_r,
$$
where $v_0\in E$ and $\nu_1,\dots,\nu_r\in\N_0$.
The vector $v_0$ and the numbers $\nu_j\in\N_0$ are uniquely determined by $v$.
\end{lemma}

\begin{proof}
For $j=1,\dots,r$ let $a_j\in\Sigma$ be the unique element such that $\al_i(a_j)=0$ for $i\ne j$ and $\al_j(a_j)$ is strictly positive and minimal.
Then $a_1,\dots,a_r$ is a basis of $V$ inside $\Sigma$, hence it generates a sublattice $\Sigma'\subset \Sigma$.
Let $E$ be a set of representatives of $\Sigma/\Sigma'$ which may be chosen such that each $v_0\in E$ lies in $C$, but for every $j=1,\dots,r$ the vector $v_0-a_j$ lies outside $C$.
It is clear that every $v$ of the form given in the statement of the lemma is in $C\cap\Sigma$.

For the converse, let $v\in C\cap \Sigma$. Then there are uniquely determined $v_0\in E$, $\nu_1,\dots,\nu_r\in\Z$ such that $v=v_0+\nu_1a_1+\dots+\nu_ra_r$.
We have to show that $\nu_1,\dots,\nu_r\ge 0$.
Assume that $\nu_j<0$.
Then
$$
0<\al_j(v)=\al_j(v_0)+\nu_j\al_j(a_j)\le\al_j(v_0)-\al_j(a_j)=\al_j(v_0-a_j)
$$
and the latter is $\le 0$, as $v_0-a_j$ lies outside $C$. This is a contradiction!
\end{proof}

We apply this lemma to $V=\bar A\otimes\Q$, the lattice $\bar A$ and the cone $A^-$.
Writing the groups multiplicatively, we obtain
\begin{align*}
\int_A\ph_u(a)a^{\la+2\rho}\,da
&=\int_{A^-}a^{\la+s}\,da\\
&= \sum_{v\in E}\ \ \sum_{\nu_1,\dots,\nu_r=0}^\infty\(va_1^{\nu_1}\cdots a_r^{\nu_r}\)^{\la+s}\\
&= \sum_{v\in E} v^\la u^{L(v)}
\frac1{1-a_1^{\la+s_1}}\cdots\frac1{1-a_r^{\la+s_r}}.
\end{align*}
Writing $a_j^{\al_j}=q^{-k_j}$ we obtain the assertion of Theorem \ref{thm2.1}.
\end{proof}

\section{Geometric zeta functions}
In this section we explain in what sense the zeta function $S_\Ga(u)$ of the last section actually counts closed geodesics.
The way it is set up, it actually counts conjugacy classes in $\Ga$ and we will show how those are connected to geometrical data.

Let $G$ be a reductive linear group over a nonarchimedian local field and let $\Ga\subset G$ be a torsion-free uniform lattice.
Further let $\CB$ denote the Bruhat-Tits building of $G$.
This is a metric space which is a union of so called \e{apartments}, each of which is isometric with $\R^r$, the latter equipped with the euclidean metric.
A \e{geodesic line} in $\CB$ is a curve which locally minimizes distances. Each geodesic necessarily lies in one apartment and via any isometry with $\R^r$, is mapped to an affine line.
A geodesic curve $c:\R\to\CB$ can therefore be normalized to speed one, i.e., such that $d(c(a),c(b))=|a-b|$, where $d$ denotes the metric of $\CB$.

The group $\Ga$ acts on $\CB$ by isometries, so $\Ga\bs \CB$ becomes a metric space and geodesics in $\Ga\bs\CB$ lift to geodesics in $\CB$. 
A geodesic curve $c:\R\to\Ga\bs\CB$ is said to be normalized or have speed one, if this holds locally.
A \e{closed geodesic} in $\Ga\bs\CB$ is a normalized geodesic curve $c:\R\to\Ga\bs\CB$ which is periodic, i.e., there exists $l(c)>0$ such that $c(t+l(c))=c(t)$ holds for every $t\in\R$.
In this case, if $z\in\CB$ is a preimage of $c(t_0)$ for some  given $t_0\in\R$  and $\tilde c$ is the unique geodesic lift of $c$ to $\CB$ such that $\tilde c(t_0)=z$, then there exists a unique $\ga\in\Ga$ such that $\ga z=\tilde c(t_0+l(c))$. In this case we say that $\ga$ \e{closes} the geodesic $c$.

\begin{proposition}
\begin{enumerate}[\rm (a)]
\item Every $\ga\in \Ga\sm\{ 1\}$ closes a geodesic in $\CB$.
\item This sets up a bijection
$$
\psi: (\Ga\sm\{ 1\})/{conjugation}\to \{\text{closed geodesics}\}/\text{homotopy}
$$
with the property that
$$
\psi([\ga^n])=\psi([\ga])^n
$$
 for every $\ga\in\Ga\sm\{ 1\}$ and every $n\in\N$.
 \item If two closed geodesics $c,c'$ in $\Ga\bs\CB$ are homotopic, then there are preimages $\tilde c,\tilde c'$ in $\CB$ which are closed by the same $\ga\in\Ga$.
\item For a given $\ga\in\Ga$ let
$$
P_\ga=\{ x\in\CB: d(x,\ga x)\text{ is minimal}\}.
$$
Then $P_\ga$ is a convex subset of the building $\CB$ which is a union of parallel geodesic lines and $\ga$ acts by translation along these geodesics.
The set $P_\ga$ equals the set of all geodesics in $\CB$ which are closed by $\ga$. 
Consequently, the closed geodesics closed by a given $\ga$ all have the same length.
\end{enumerate}
\end{proposition}

\begin{proof}
(a)
Let $\Ga\in\Ga\sm\{1\}$.
As $\Ga$ is torsion-free, the element $\ga$ has no fixed point in $\CB$.
As $\ga$ preserves the simplicial structure on $\CB$, the function $p\mapsto d(p,\ga p)$ attains a minimal value $m>0$.
The set $P= P_\ga$ defined above is therefore well-defined and non-empty.
We first claim that $P$ is a union of $\ga$-stable geodesic lines on each of which $\ga$ acts by a translation. 
Let $p\in P$ and let $z$ be on the line segment between $p$ and $\ga p$.
Then we have
\begin{align*}
d(z,\ga z)&\le d(z,\ga p)+d(\ga p,\ga z)\\
&= d(z,\ga p)+d(p,z)\\
&=d(p,\ga p)=m.
\end{align*}
As $m$ is minimal, we have equality and the geodesic from $z$ to $\ga z$ is the composite of $\ol{z,\ga p}$ and $\ol{\ga p,\ga z}$, which means that the line segment $\ol{p,\ga z}$ is geodesic. 
We repeat this construction with $z$ in place of $p$ and in this way extend $\ol{p,\ga p}$ to a geodesic line which is preserved by $\ga$ and on which $\ga$ acts by translation.
This proves (a) and parts of (d)

(b)
As $\Ga$ is the fundamental group of $B_\Ga=\Ga\bs B$ we have a natural bijection
$$
\Ga/\text{conjugation}\to [S^1,B_\Ga],
$$
where the right hand side is the set of free homotopy classes of loops.
Also, there is a trivial injection
$$
\{\text{closed geodesics}\}/\text{homotopy}\hookrightarrow [S^1,B_\Ga].
$$
These maps compose to give the desired injective map 
$$
\psi:\{\text{closed geodesics}\}/\text{homotopy}\hookrightarrow \Ga/\text{conjugation}.
$$
By the first part, the image of this map is $\Ga\sm\{ 1\}/\text{conjugation}$.

(c) Let $\ga$ and $\ga'$ be elements of $\Ga$ closing some preimages $\tilde c$ and $\tilde c'$ of $c$ and $c'$.
By (b), the elements $\ga$ and $\ga'$ must be conjugate, which means that the preimages $\tilde c$ and $\tilde c'$ can be chosen in such a way that $\ga=\ga'$.

(d)
We already know that $P_\ga$ is a union of geodesic lines.
By construction, for $p\in P$, the convex hull $L_p$ of the set $\ga^\Z p$ is the unique geodesic line closed by $\ga$ and containing $p$.
Now let $q$ be another point of $P$, then the distance of any point on the geodesic line $L_q$ to any point on the line  $L_p$ is bounded, which can only happen if the two geodesics $L_p$ and $L_q$ lie in a common apartment and are parallel in that apartment.
The convex hull of these two lines is preserved by $\ga$ and as $\ga$ is a translation on both lines, it is a translation on this convex hull.
This proves the convexity of $P_\ga$.

Now finally, let $L$ be any geodesic which is closed by $\ga$ and let $z$ be a point of $L$.
Let $p$ be a point of $P_\ga$. Then again the lines $L$ and $L_p$ are parallel and thus lie in the same apartment, $\ga$ must act by the same translation and thus $L$ belongs to $P_\ga$.
\end{proof}

\begin{lemma}
Assume that $\Ga$ is torsion-free and let $\ga\in\Ga$.
Let $S\subset\CB$ be a $\Ga$-stable affine subset.
Then there exists an origin $0$ in $S$, a linear orthogonal transformation $T:S\to S$  and a point $b\in S\sm\{ 0\}$ with $Tb=b$ such that $\ga x=Tx+b$.
\end{lemma}

\begin{proof}
As $\ga$ fixes the euclidean structure on $S$, it acts, after choosing an arbitrary origin, as $\ga x=Tx+b$ for some linear orthogonal $T$ and some $b\in S$.
Let $U$ be the eigenspace of the eigenvalue $1$ for $T$ and let $V$ be its orthocomplement.
We have the orthodecomposition $b=b_U+b_V$.
As $1-T:V\to V$ is surjective, there exists $v_0\in V$ with $(1-T)v_0=b_V$, or $\ga v_0-b=Tv_0=v_0-b_V$, which amounts to $\ga v_0=v_0+b_U$.
Since $\Ga$ is torsion-free, $\ga$ fixes no point in $\CB$ and so $b_U\ne 0$.
Relocating the zero to the point $v_0$ gives the claim.
\end{proof}

An element $g$ of $G$ is called \e{admissible}, if there exists a parabolic group $P=LN$ defined over $F$, such that $g$ lies in $A_L^\reg M_L$.
Here, for a given torus $A$, the set $A^\reg$ is the set of \e{regular elements}, i.e., the set of all $a\in A$ such that the centralizer $G_a$ of $a$ in $G$ equals the pointwise centralizer of $A$.
Then the group $A_L^\reg M_L$ has finite index in $L$ and as there are only finitely many conjugacy classes of parabolic subgroups, there exists $N\in \N$ such that $g^N$ is admissible for every semisimple, non-elliptic element $g$.
A subgroup $\Ga\subset G$ is called admissible, if every $\ga\in \Ga\sm\{1\}$ is admissible.

For simplicity of exposition, we will now assume that $G$ is simple, which implies that the Bruhat-Tits building $\CB$ is a simplicial complex.
Let $r\in\N$.
An \e{$r$-dimensional path} is a sequence $ \dots,S_{-1},S_0,S_1,\dots$ of $r$-dimensional simplices such that $S_j$ and $S_{j+1}$ have a common face of dimension $r-1$ for each $j\in\Z$.
We say that the path is \e{geodesic}, if there exists a geodesic line $L$ with $L\cup \mathring S_j\ne\emptyset$ for every $j\in\Z$.
Here $\mathring S$ denotes the interior of the simplex $S$.
If this is the case, then all $S_j$ lie in a common apartment $\CA$. We say that a given $\ga\in\Ga$ \e{closes} the path $(S_j)$ if $\ga S_j=S_{j+n}$ holds for all $j\in\Z$ and some $n\in\N$.
If this is the case, then $\ga$ stabilizes  the union of the $S_j$.
This union lies in a common apartment, so it carries an euclidean structure.
Therefore, after fixing an origin in $S$, the element $\ga$ acts as $\ga x=Tx+b$, where $T$ is linear orthogonal and $b\in S\sm\{ 0\}$.
Actually, $T$ fixes $b$ and thus can be considered an orthogonal transformation of the orthogonal space of $b$.

\section{$\PGL_3$}\label{SecPGL}
In this section we explain the connection between the current zeta function in several variables and the zeta functions occurring in a generalized Ihara formula in the paper \cite{KLW}.

The vertices of the building of $G=\PGL_3(F)$ are parametrized by homothety classes of $\CO$-lattices in $F^3$.
Let $v_F:F^\times\to\Z$ denote the valuation of the local field $F$.
The group $G$ acts transitively on the latter, but the index three subgroup $G'$ of all $g\in G$ with $v_F(\det(g))\equiv 0\mod(3)$ has three orbits, which are given by the representatives 
\begin{align*}
L_0&=\sp{e_1,e_2,e_3}\\
L_1&=\sp{e_1,e_2,\pi e_3}\\
L_2&=\sp{e_1,\pi e_2,\pi e_3}
\end{align*}
We say a vertex $v$ is of \e{type} $j\mod(3)$, if it is in the $G'$-orbit of $L_j$. 
We assume from now on, that $\Ga$ is contained in $G'$, so that $\Ga$ preserves types of vertices.

A  geodesic $c$ in $\CB$ or $\Ga\bs \CB$ is called \e{rational}, if it contains a point of the zero skeleton and is called \e{integral}, if it is contained in the 1-skeleton of $\CB$ or $\Ga\bs \CB$. Every integral geodesic is rational.
The vertices on an integral geodesic either have consecutive types $0,1,2$ or $2,1,0$. In the first case, the geodesic is called \e{positive} in the latter it is \e{negative}.
The inverse of a positive geodesic is negative and vice versa.
A geodesic parallel to an integral positive geodesic is also called positive.

An element $\ga\in\Ga\sm\{1\}$ is called positive, if it closes a positive geodesic, i.e., if for one and thus every point $p$ in $P_\ga$ the geodesic line through $\ga^\Z p$ is positive.

Let $C_{int}(\Ga)$ denote the set of all integral geodesics in $\Ga\bs\CB$.
Then every element of $C_{int}(\Ga)$ is actually closed, as we show below.

For $G=\PGL_3(F)$ there are three different classes of proper parabolics: $P_0$ is the group of all upper triangular matrices, 
$$
P_1=\left(\begin{array}{ccc} &  & * \\ &  &  \\0 & 0 & \end{array}\right)\quad\text{and}\quad P_2=\left(\begin{array}{ccc} &  & * \\0 &  &  \\0 &  & \end{array}\right).
$$
We write $P_j=L_jN_j$ for the Levi decomposition and we fix subgroups $M_jA_j\subset L_j$  as in Section \ref{sec1}.
We choose $A_0$ to be the group of all diagonal matrices, $A_1$ to be the subgroup of all matrices of the form $\diag(a,a,b)$ with $a,b\in F$, and $A_2$ to consist of all matrices of the form $\diag(a,b,b)$.

An element $\diag(a,b,c)$ of $A_0$ is called \e{strongly regular}, if the absolute values $|a|,|b|,|c|$ are all different.

The following Euler product can be viewed as one posible version of the \e{Ruelle zeta function} of compact Riemann surface. In the case of a higher-dimensional building, the zeta functions have to take the fact into account, that one can have geodesics, which in a given apartment point into different directions. 
This fact is taken care of by the fact that generally one deals with several variable zeta functions.

Restricting the variables, however, one gets single variable zeta functions, and sometimes these turn out to be representable by Euler products, in contrast to the several variable case.
The following zeta function only considers geodesics which point in one given direction, therefore the length suffices to characterize them and one gets a single variable zeta function.

\begin{definition}
Let 
$$
Z_{1,+}(u)=\prod_{c}\(1-u^{l(c)}\),
$$
where the product is extended over all closed integral positive primitive geodesics in $\Ga\bs\CB$.
Here a closed geodesic $c$ is called \e{primitive}, if it is not a power of a shorter one.
Note that $l(c)$ here denotes the length of the closed geodesic $c$.
\end{definition}

The reader may note, that this definition differs from Ihara's and others in that we consider the Euler factors with a positive exponent while other authors would prefer $1/Z_{1,+}(u)$ instead.
There is, however, good reason for this: firstly, as the next lemma shows, the function becomes a polynomial in this way, which is slightly easier to handle than the inverse of a polynomial. Next, Ihara's zeta function is the $p$-adic version of Selberg's zeta function and Selberg chose our current normalization, so we stuck with his notation.
There is a deep reason for the sign in Selberg's paper, which is explained in \cite{Geom1}: It emerges that the exponent prescribed by the trace formula is an Euler number, in Selberg's original case the Euler number of a point, but generally the Euler number of a locally symmetric space which can be negative, as happens in the example of $\SL_2\times\SL_2$.

\begin{lemma}
The infinite product $Z_{1,+}(u)$ converges absolutely to a polynomial in $u$, when $|u|$ is small.
\end{lemma}

As usual, we will also write $Z_{1,+}(u)$ for the polynomial which is the limit.

\begin{proof}
Let $N$ be the number of edges in the finite building $\Ga\bs\CB$.
Then there are no more than $N^m$ geodesic paths of length $m$.
Therefore,
$$
\sum_c |u|^{l(c)}\le\sum_{m=1}^\infty N^mu^m,
$$
which converges for $|u|<1/N$ and so the product $Z_{1,+}(u)$ also converges in that region.
Having settled convergence, let $V=\bigoplus_e\C e$ be the formal complex vector space generated by all edges $e$ of $\Ga\bs\CB$.
define a linear operator $T:V\to V$ by
$$
T(e)=\sum_{e'}e',
$$
where the sum ranges over all edges $e'$ connected to $e$ such that the path $ee'$ is positive.
We equip $V$ with the inner product $\sp{\ ,\ }$ defined by making the edges an orthonormal basis, i.e., by
$$
\sp{e,e'}=\begin{cases}1&e=e'\\ 0,& e\ne e'.\end{cases}
$$
For $n\in \N$ consider the $n$-th iteration $T^n$ of $T$. It's trace is
$$
\tr T^n=\sum_e\sp{T^ne,e}.
$$ 
If one imagines the action of $T$ as sending a potential from $e$ to neighboring edges $e'$ along positive directions, it becomes clear, that $\sp{T^ne,e}$ can only be non-zero, if $e$ lies on a positive path of length $n$.
Therefore one gets
$$
\tr T^n=\sum_{c:l(c)=n}l(c_0),
$$
where the sum runs over all positive paths $c$ of length $n$ and $c_0$ is the unique primitive positive path such that $c$ is a power of $c_0$.
We keep the notation that $c_0$ denotes a positive primitive closed geodesic, also $\sum_c$ will denote the sum over all positive closed geodesics $c$  and $\sum_{c_0}$ the sum over all positive primitive closed geodesics.
Note that if $c=c_0^m$, then $l(c)=ml(c_0)$, i.e., $m=l(c)/l(c_0)$.
We then get
\begin{align*}
Z_{1,+}(u)&=\prod_{c_0}\(1-u^{l(c_0)}\)=\exp\(\sum_{c_0}\log\(1-u^{l(c_0)}\)\)\\
&=\exp\(-\sum_{c_0}\sum_{m=1}^\infty\frac{u^{l(c_0)m}}m\)=\exp\(-\sum_{c}\frac{u^{l(c)}}{l(c)}l(c_0)\)\\
&=\exp\(-\sum_{n=1}^\infty\frac{u^{n}}{n}\sum_{c:l(c)=n}l(c_0)\)\\
&=\exp\(-\sum_{n=1}^\infty\frac{u^{n}}{n}\tr(T^n)\)=\exp\(\log(1-uT)\)\\
&=\det(1-uT).
\end{align*}
Here all infinite sums converge when $|u|$ is small enough.
\end{proof}

\begin{definition}
Let
$$
Z_{2,+}(u)=\prod_p\(1-u^{l(p)}\),
$$
where the product ranges over all positive primitive closed geodesic paths in $\Ga\bs\CB$ of dimension 2 and the length is the number of chambers such a path contains.

Similar to the above, it can be shown that $Z_{2,+}(u)$ converges to a polynomial when $|u|$ is small enough.
\end{definition}

We now show how the different zeta functions we have defined, are interrelated. 
The function $S_{\Ga, P_1}(u)$, as defined in Section \ref{Sec2}, uses group theoretical input from the groups $\Ga$ and $G$. The connection to geometric zeta functions defined by geometric data on the quotient building $\Ga\bs\CB$ is the fact, that each closed geodesic $c$ gets closed by some $\ga\in\Ga$.
In the special situation of the parabolic $P_1$, the rank of $A_1$ is one and therefore the function $S_{\ga,P_1}$ actually is a single variable function.
We now show how it is related to the geometric zeta functions defined via  rational geodesics and paths of higher dimensions $Z_{1,+}$ and $Z_{2,+}$.

\begin{theorem}\label{thm4.4}
After replacing the group $\Ga$ with a finite index subgroup, we have the identity of rational functions,
$$
\frac{Z_{2,+}(-u)}{Z_{1,+}(u^2)}=\exp\(- \int_0^u  S_{\Ga,P_1}(z)\,dz\).
$$
Or, otherwise stated, $S_{\Ga,P_1}(u)=\frac{F'}F(u)$, where $F(u)=\frac{Z_{1,+}(u^2)}{Z_{2,+}(u)}$.
\end{theorem}

For the proof of the theorem, we will need the following lemma.

\begin{lemma}
After replacing the group $\Ga$ with a finite index subgroup, we can assume $\Ga$ to be regular in the sense that every $\ga\in\Ga\sm\{1\}$ lies in the regular set $G^\reg$.
\end{lemma}

\begin{proof}
By Margulis's arithmeticity result we know that $\Ga$ is arithmetic, so there exists a global field $\kappa$, of which $F$ is a local completion, and a division algebra $M$ over $\kappa$ of degree 3, which splits at $F$, such that $\Ga$ is commensurable with the image of $M(\La)^\times$ in $G(F)$, where $\La$ is some order in $\kappa$.
Replacing $\Ga$ by a finite index subgroup, we may assume that $\Ga$ lies in that image.
For a given $\ga\in\Ga\sm\{1\}$ fix a preimage $\tilde\ga\in M(\La)$.
The centralizer $M_{\tilde\ga}$ of $\tilde\ga$ in $M$ is a proper subalgebra, whose degree must divide the degree of $M$, which is a prime, therefore the degree of $M_{\tilde\ga}$ is one, so $M_{\tilde\ga}$ is a field, hence commutative and so is $G_\ga$ which is the image of $M_{\tilde\ga}(F)$.
Therefore $\ga$ is regular.
\end{proof}

\begin{proof}[Proof of the theorem]
By the lemma we can assume $\Ga$ to be regular.
In this case, each centralizer $G_\ga$ is a torus, so for $\ga\in\CE_{P_1}(\Ga)$ the group $\Ga_\ga$ will be isomorphic to $\Z$, so that $\chi_1(\Ga_\ga)=1$.
By the normalizations of Haar measures we see that $\la_\ga=l(\ga_0)$, where $\ga_0$ is the underlying primitive element.
Thus the Selberg zeta function equals
$$
S_{\Ga,P_1}(u)=\sum_{[\ga]\in\CE_{P_1}(\Ga)}l(\ga_0) u^{l(\ga)}.
$$
Note that in this particular situation, as the rank of $A_1$ equals one, we indeed have $u^{l(\ga)}=u^{L(\ga)}$.
So that for small enough $u$,
\begin{align*}
\exp\int_0^u S_{\Ga,P_1}(z)\,dz&= \exp\sum_{[\ga]}l(\ga_0)\frac1{l(\ga)}u^{l(\ga)}\\
&=\exp\sum_{[\ga_0]}\sum_{n=1}^\infty\frac{u^{l(\ga_0)n}}n\\
&= \exp\(-\sum_{\ga_0}\log(1-u^{l(\ga_0)})\)\\
&=\prod_{[\ga_0]\in\CE_{P_1,prim}(\Ga)}\(1-u^{l(\ga_0)}\)^{-1},
\end{align*}
where the product extends over all primitive elements in $\CE_{P_1}(\Ga)$.
Taking inverses, it remains to show
$$
\frac{Z_{1,+}(u^2)}{Z_{2,+}(-u)}=\prod_{[\ga_0]\in\CE_{P_1,prim}(\Ga)}\(1-u^{l(\ga_0)}\).
$$
To prove this, we will make use of the following phenomenon: If $p$ is a closed gallery path in $\Ga\bs\CB$, then the boundary of $p$ consists of two or one closed integral geodesics, depending on whether $p$ is orientable or not.
In the orientable case, the length of $p$ will be twice the length of either of the geodesics, so the contribution of $p$ to the product $Z_{2,+}(-u)$ will equal the contribution of either of the two geodesics in $Z_{1,+}(u^2)$.
The minus sign will not play a role as the length of the gallery path is even.
In the non-orientable case, one gets only one closed geodesic and this has the same length as $p$, which is an odd number and one gets the contribution $\frac{1-u^{2l(\ga)}}{1+u^{l(\ga)}}=1-u^{l(\ga)}$.
This kind of reduction is used in the sequel.

Start with a positive closed primitive integral geodesic $c$, choose a preimage $\tilde c$ in $\CB$ and let $\ga\in\Ga$ be an element closing $\tilde c$.
We first deal with the case when $\ga$ is not a primitive element of $\Ga$, so $\ga=\ga_0^n$ for some $n\in\N$, $n\ge 2$.
Then the geodesic $\tilde c$ lies in $P_\ga$, but not in $P_{\ga_0}$.
So, $\tilde c$ must be parallel to a geodesic $\tilde c_0$ in $P_{\ga_0}$.
The convex hull of $\tilde c$ and $\tilde c_0$ is a strip consisting of finitely many galleries and the one adjacent to $\tilde c$ is closed by $\ga$ and no lower power of $\ga_0$.
Therefore, the factor of $\tilde c$ cancels with the factor of this particular gallery.
It follows that we only have to consider geodesics closed by primitive elements of $\Ga$.

{\it First case.} Assume that $G_\ga$ is a split torus. 
Then $\ga$ induces a translation on the apartment $S$ attached to $G_\ga$.
This apartment therefore lies in $P_\ga$.
The set $P_\ga$ is a union of parallel geodesic lines, so every $x_0\in P_\ga$ lies in a unique geodesic line $L=c(\R)$ fully contained in $P_\ga$, where the geodesic curve $c:\R\to P_\ga$ with $c(0)=x_0$ is uniquely determined up to orientation, i.e., up to replacing $c(t)$ by $c(-t)$.
We assume that orientations have been chosen in a compatible way so that we obtain an action of $\R$ on the set $P_\ga$ given by $(t,x_0)\mapsto c(t)$. We call this the \e{geodesic action} of $\R$.
We consider the quotient of $P_\ga$ by the geodesic action and  see that $P_\ga/\R$ is a tree which contains a line $L$.
We claim that the structure of this tree is as such that $P_\ga/\R\sm L$ is a union of disjoint finite trees.
This is a consequence of the fact that $\Ga_\ga\bs P_\ga$ is compact.
So, modulo geodesic gallery paths, one can reduce each of the finite trees to a point and so reduce $P_\ga$ to one apartment $S$.
The image of $S$ in $\Ga\bs \CB$ is a union of closed geodesics or of closed gallery paths and both occur in the same number, so that they cancel in the quotient 
$\frac{Z_{1,+}(u^2)}{Z_{2,+}(-u)}$.

{\it Second case.}
If $G_\ga$ is a non-split torus,
then $P_\ga$ will not contain an apartment.
Then $P_\ga/\R$ is compact 
We have two possible situations.
The first is that $P_{\ga}$ contains an integral geodesic, so we can reduce to that one and get one remaining contribution of the form $(1-u^{2l(\ga)})$.
If $P_\ga$ does not contain an integral geodesic, this implies that $P_\ga$ is a single line going through the interior of a gallery path, which is not closed by $\ga$, but by $\ga^2$.
In the quotient, this is exactly the non-orientable case and the argument given above proves Theorem \ref{thm4.4}.
\end{proof}

{\bf Remark:}
In \cite{KL}, the authors also study $\frac{Z_{1,+}(u^2)}{Z_{2,+}(-u)}$ by taking the logarithm derivative and show that it can be expressed as a group zeta function, which indeed equals the right hand side of the above theorem. Their computations are much more complicated but include that case that $\Gamma$ is not regular.

\section{Riemann Hypothesis}
The complex $\CB_\Gamma = \Gamma \backslash \CB$ is called \emph{Ramanujan} if all irreducible unramified infinite dimensional subrepresentations of $L^2(\Gamma \backslash G)$ are tempered. See \cite{Li} and \cite{LSV} for details. When $G=\PGL_2(F)$, $\CB_\Gamma$ is a finite regular graph.
It was first pointed out by \cite{Sunada} that the non-trivial poles of Ihara zeta function $Z(\CB_\Gamma,u)$  of $\CB_\Gamma$ have absolute values equal to $q^{-1/2}$ if and only if $\CB_\Gamma$ is a Ramanujan graph. The first condition is called the Riemann hypothesis of $Z(\CB_\Gamma,u)$ since if we replace $u$ by $q^{-s}$, the condition becomes that all non-trivial poles of $Z(\CB_\Gamma,q^{-s})$ lie on Re$(s)=\frac{1}{2}$. We remark that $Z(\CB_\Gamma,u)$ satisfies the Riemann hypothesis if and only if $\CB_\Gamma$ is a Ramanujan graph.

We shall give an analogue of the above statement for $G=\PGL_3(F)$.
Recall that $Z_{1,+}(u)$ and  $Z_{2,+}(u)$ are polynomials so that $ Z_{1,+}(u)=\det(I-L_E u)$ for some parahoric Hecke operator $L_E$; $ Z_{2,+}(u)=\det(I-L_B u)$ for some Iwahori Hecke operator $L_B$ \cite{KLW}. Given a smooth unramified representation $V$ of $G$, consider 
$$ Q(V,u)=\frac{\det(I+ L_B u)}{\det(I-L_E u^2)}$$
where the determinant is taken over the spaces of parahoric and Iwahori fixed vectors of $V$ respectively. Then we have 
$$ \frac{Z_{2,+}( -u)}{Z_{1,+}(u^2)} = \prod_{V} Q(V,u)^{m_V}$$
where $V$ runs through all irreducible unitary Iwahori-spherical subrepresentations of $L^2(\Gamma \backslash G)$ and $m_V$ is its multiplicity.
From Table 1 and Table 2 in \cite{KLW}, we have \\
\begin{itemize}
\item[(a)] If $V$ is a principal series representation, then $Q(V,u)=1$.

\item[(b)] If $V$ is the trivial representation twisted by a cubic unramified character $\chi$ of $F$, then $Q(V,u)=\frac{1}{1- q \chi(\pi) u}$ and $m_V=1$. 

\item[(c)] If $V$ is the Steinberg representation twisted by a cubic unramified character $\chi$ of $F$, then $Q(V,u)={1- \chi(\pi) u}$ and $m_V= \chi(X_\Gamma)-1$.

\item[(d)] If $V$ the irreducible subrepresentation of $\Ind(\chi|~|^{-1/2}, \chi|~|^{1/2}, \chi^{-2})$, where $\chi$ is
an unramified unitary character of $F^{\times}$.
Then $Q(V,u)={1}/(1- q^{1/2} \chi(\pi) u)$. Moreover, $V$ is not tempered.

\item[(e)] The irreducible subrepresentation of
$\Ind(\chi|~|^{1/2}, \chi|~|^{-1/2}, \chi^{-2})$, where $\chi$ is
an unramified unitary character of $F^{\times}$. 
Then we have $Q(V,u)={1- q^{1/2} \chi(\pi) u}$. 
\end{itemize}

We summarize the above in the following theorem
\begin{theorem}
With the above notation we have
$$
\frac{Z_{2,+}( -u)}{Z_{1,+}(u^2)} = \frac{(1-u^3)^{\chi-1}P_1(u)}{(1-q^3u^3)P_2(u)}
$$ 
where $P_1(u)=\prod_\alpha (1-\alpha u)$ and $P_2(u)=\prod_\beta (1-\beta u)$ with $|\alpha|=|\beta|=q^{1/2}$.
\end{theorem}

\begin{corollary}
When $\CB_\Gamma$ is a Ramanujan complex, then 
$$ \frac{Z_{2,+}( -u)}{Z_{1,+}(u^2)}= (1-u^3)^{\chi} \frac{P_1(u)}{(1-u^3)(1-q^3u^3)},$$ where $P_1(u)=\prod_\alpha (1-\alpha u)$ with $|\alpha|=q^{1/2}$ of degree $N_1-3N_0+6$. Here $N_i$ is the number of $i$-simplex in $\CB_\Gamma$. In this case, we say the complex zeta functions of $\CB_\Gamma$ satisfy the Riemann hypothesis. 
\end{corollary}

Anton Deitmar \\
{\small Mathematisches Institut,
Auf der Morgenstelle 10,
72076 T\"ubingen, Germany,
\tt deitmar@uni-tuebingen.de}

Ming-Hsuan Kang \\
{\small Department of Applied Mathematics, National Chiao-Tung University,
Hsinchu, Taiwan,
\tt mhkang@math.nctu.edu.tw}


\begin{bibdiv} \begin{biblist}

\bib{borel-lingroups}{book}{
   author={Borel, Armand},
   title={Linear algebraic groups},
   series={Graduate Texts in Mathematics},
   volume={126},
   edition={2},
   publisher={Springer-Verlag},
   place={New York},
   date={1991},
   pages={xii+288},
   isbn={0-387-97370-2},
   doi={10.1007/978-1-4612-0941-6},
}

\bib{borel}{book}{
   author={Borel, Armand},
   title={Introduction aux groupes arithm\'etiques},
   language={French},
   series={Publications de l'Institut de Math\'ematique de l'Universit\'e de
   Strasbourg, XV. Actualit\'es Scientifiques et Industrielles, No. 1341},
   publisher={Hermann},
   place={Paris},
   date={1969},
   pages={125},
}

\bib{Borel-Wallach}{book}{
   author={Borel, A.},
   author={Wallach, N.},
   title={Continuous cohomology, discrete subgroups, and representations of
   reductive groups},
   series={Mathematical Surveys and Monographs},
   volume={67},
   edition={2},
   publisher={American Mathematical Society},
   place={Providence, RI},
   date={2000},
   pages={xviii+260},
   isbn={0-8218-0851-6},
}

\bib{Cartier}{article}{
   author={Cartier, P.},
   title={Representations of $p$-adic groups: a survey},
   conference={
      title={Automorphic forms, representations and $L$-functions},
      address={Proc. Sympos. Pure Math., Oregon State Univ., Corvallis,
      Ore.},
      date={1977},
   },
   book={
      series={Proc. Sympos. Pure Math., XXXIII},
      publisher={Amer. Math. Soc., Providence, R.I.},
   },
   date={1979},
   pages={111--155},
}

\bib{Geom1}{article}{
   author={Deitmar, Anton},
   title={Geometric zeta functions of locally symmetric spaces},
   journal={Amer. J. Math.},
   volume={122},
   date={2000},
   number={5},
   pages={887--926},
   issn={0002-9327},
}

\bib{GAFA}{article}{
   author={Deitmar, A.},
   title={A prime geodesic theorem for higher rank spaces},
   journal={Geom. Funct. Anal.},
   volume={14},
   date={2004},
   number={6},
   pages={1238--1266},
   issn={1016-443X},
   doi={10.1007/s00039-004-0490-7},
}

\bib{padlef}{article}{
   author={Deitmar, Anton},
   title={Lefschetz formulae for $p$-adic groups},
   journal={Chin. Ann. Math. Ser. B},
   volume={28},
   date={2007},
   number={4},
   pages={463--474},
   issn={0252-9599},
   doi={10.1007/s11401-005-0234-5},
}

\bib{Ihara}{article}{
   author={Ihara, Yasutaka},
   title={On discrete subgroups of the two by two projective linear group
   over ${\germ p}$-adic fields},
   journal={J. Math. Soc. Japan},
   volume={18},
   date={1966},
   pages={219--235},
   issn={0025-5645},
}

\bib{KL}{article}{
   author={Kang, Ming-Hsuan},
   author={Li, Wen-Ching Winnie},
   title={The zeta functions of complexes from ${\rm PGL}(3)$},
   journal={Advances in Mathematics},
   volume={256},
   pages={46--103},
   date={2014}
}

\bib{KLW}{article}{
   author={Kang, Ming-Hsuan},
   author={Li, Wen-Ching Winnie},
   author={Wang, Chian-Jen},
   title={The zeta functions of complexes from ${\rm PGL}(3)$: a
   representation-theoretic approach},
   journal={Israel J. Math.},
   volume={177},
   date={2010},
   pages={335--348},
   issn={0021-2172},
   doi={10.1007/s11856-010-0049-2},
}

\bib{Kottwitz}{article}{
   author={Kottwitz, Robert E.},
   title={Tamagawa numbers},
   journal={Ann. of Math. (2)},
   volume={127},
   date={1988},
   number={3},
   pages={629--646},
   issn={0003-486X},
   doi={10.2307/2007007},
}


\bib{Li}{article}{
   author={Li, Wen-Ching Winnie},
   title={Ramanujan hypergraphs},
   journal={Geom. Funct. Anal.},
   volume={14},
   date={2004},
   number={2},
   pages={380--399},
   issn={1016-443X},
   doi={10.1007/s00039-004-0461-z},
}

\bib{LSV}{article}{
   author={Lubotzky, Alexander},
   author={Samuels, Beth},
   author={Vishne, Uzi},
   title={Ramanujan complexes of type $A_d$},
   note={Probability in mathematics},
   journal={Israel J. Math.},
   volume={149},
   date={2005},
   pages={267--299},
   issn={0021-2172},
   doi={10.1007/BF02772543},
}



\bib{Sunada}{article}{
   author={Sunada, Toshikazu},
   title={$L$-functions in geometry and some applications},
   conference={
      title={Curvature and topology of Riemannian manifolds},
      address={Katata},
      date={1985},
   },
   book={
      series={Lecture Notes in Math.},
      volume={1201},
      publisher={Springer},
      place={Berlin},
   },
   date={1986},
   pages={266--284},
   doi={10.1007/BFb0075662},
}

\bib{Tits}{article}{
   author={Tits, J.},
   title={Reductive groups over local fields},
   conference={
      title={Automorphic forms, representations and $L$-functions},
      address={Proc. Sympos. Pure Math., Oregon State Univ., Corvallis,
      Ore.},
      date={1977},
   },
   book={
      series={Proc. Sympos. Pure Math., XXXIII},
      publisher={Amer. Math. Soc., Providence, R.I.},
   },
   date={1979},
   pages={29--69},
}

\bib{Wolf}{article}{
   author={Wolf, Joseph A.},
   title={Discrete groups, symmetric spaces, and global holonomy},
   journal={Amer. J. Math.},
   volume={84},
   date={1962},
   pages={527--542},
   issn={0002-9327},
}

\end{biblist} \end{bibdiv}
\end{document}